\documentclass{amsart}
\usepackage{amsmath,amsthm,amssymb}
\usepackage{graphicx}

\newtheorem{theorem}{Theorem}[section]
\newtheorem{proposition}[theorem]{Proposition}
\newtheorem{lemma}[theorem]{Lemma}
\newtheorem{corollary}[theorem]{Corollary}

\theoremstyle{definition}
\newtheorem{definition}[theorem]{Definition}
\newtheorem{remark}[theorem]{Remark}

\begin{document}

\title[Gluck twisting 4-manifolds with odd intersection form]{Gluck twisting 4-manifolds with\\ odd intersection form}

\author[AKBULUT and YASUI]{Selman Akbulut and Kouichi Yasui}
\thanks{The first named author is partially supported by NSF grant DMS 0905917, and the second named author was partially supported by KAKENHI 23840027.}
\date{March 8, 2013}
\address{Department~of~Mathematics, Michigan State University, E.Lansing, MI, 48824, USA}
\email{akbulut@math.msu.edu}

\address{Department~of~Mathematics, Graduate School~of~Science, Hiroshima~University, 1-3-1 Kagamiyama, Higashi-Hiroshima, Hiroshima, 739-8526, Japan}\email{kyasui@hiroshima-u.ac.jp}


\begin{abstract}

We give a simple criterion when a Gluck twisting an odd smooth $4$-manifold along a $2$-sphere $S\subset X$ does not change its diffeomorphism type. We obtain this by  handlebody techniques and plug twisting operation, getting a slightly stronger version of the known fact that Gluck twisting of a $2$-sphere $S\subset X$ of a compact smooth $4$-manifold with an odd spherical class, in the complement of $S$, does not change the diffeomorphism type of $X$. This is the best possible result on Gluck twisting manifolds with odd homology classes.


 \end{abstract}

\maketitle

\vspace{-.2in}

\section{Introduction}\label{sec:intro} An interesting problem of 4-dimensional topology is to determine whether $S^4$ has an exotic smooth structure. The Gluck twisting operation ~\cite{Gl2} has been regarded as a potential method to produce an exotic $S^4$ (cf.\ \cite{GS}). This is the operation of removing an embedded $S^2\times D^2$ from a 4-manifold and regluing it via the non-trivial diffeomorphism of $S^2\times S^1$. In \cite{A3} by Gluck twisting operation, an exotic pair of non-orientable 4-manifolds was constructed. On the other hand, many families of knotted $2$-spheres do not change $S^4$ by Gluck twisting (e.g.\  \cite{Gl2}, \cite{Gor},  \cite{A4}, \cite{G2}, \cite{NS}). Indeed, no exotic pair of orientable 4-manifolds, which are related each other by Gluck twist, known to exits.
%
%
%
%
It is thus interesting to see whether Gluck twists can produce an exotic pair of orientable 4-manifolds. 


\vspace{.05in}

For a smooth $4$-manifold $X$ and  a smoothly imbedded $2$-sphere with trivial normal bundle $S\subset X$, let $\nu(S)\cong S^2\times D^2$ denote the tubular neighborhood of $S$. Let ${X}_{{S}}^\circ$  be the manifold obtained by surgering $X$ along $S$, that is
$${X}_{{S}}^\circ= (X-\nu(S))\cup D^{3}\times S^{1}.$$
%
%
We call a second homology class $\alpha$ of $X$ is sherical, if $\alpha$ is representd by a continuous map of the 2-sphere. In this paper, by using  handlebody techniques and plug twisting operation, we give an alternative proof of the known result (Corollary~\ref{intro:cor}). In fact, this approach gives the slightly stronger result below. 
\begin{theorem}\label{intro:thm} Let $X$ be any compact connected smooth $4$-manifold, and $S\subset X$ be a smoothly embedded $2$-sphere with trivial normal bundle. Then, Gluck twisting $X$ along $S$ does not change the diffeomorphism type of $X$, provided there is a $2$-dimensional spherical homology class in $X_{S}^\circ$ with odd self intersection. \end{theorem}

\newpage
\begin{remark} Note that this theorem is the best possible result about Gluck twisting odd manifolds . This is because \cite{A3} implies that one can get an example of a $2$-sphere in an odd manifold $S\subset X$, where Gluck twisting $X$ along $S$ can change its diffeomorphism type (consequently $X_{S}^{\circ}$ does not have any odd spherical class). Note that the orientibility is not part of the hypothesis of our theorem.   

\vspace{.05in}
Also note that the condition of having a spherical class in $H_{2}(X_{S}^\circ; \mathbb{Z})$ with odd self intersection is equivalent to having a class of odd self intersection in the cohomology group  $H^{2}(X_{S}^\circ, \Lambda)$ with coefficients, where $\Lambda$ is the group ring $\mathbb{Z}[\pi_{1}(X_{S}^\circ)]$, since $$H^{2}(X_{S}^\circ, \Lambda)\cong H^{2}_{c}(\tilde{X}_{S}^\circ, \mathbb{Z})\cong H_{2}(\tilde{X}_{S}^\circ, \mathbb{Z})\cong \pi_{2}(X_{S}^\circ),$$
where $\tilde{X}_{S}^\circ$ denotes the universal cover, and $H^{2}_{c}$  cohomology with compact support.\end{remark}
\vspace{.05in}

\begin{corollary}[cf.~\cite{GS}]\label{intro:cor} Let $X$ and $S$ be as in Theorem~\ref{intro:thm}. If $X-\nu(S)$ has a simply connected codimension 0 submanifold with odd intersection form, then Gluck twisting $X$ along $S$ does not change the diffeomorphism type of $X$. 
\end{corollary}


\textbf{Acknowledgements.} This work was done during the second author's stay at Michigan State University in March 2012 and March 2013. He would like to thank their hospitality. The authors would like to thank Robert Gompf for pointing out an error in an earlier draft of this paper. 

\section{Gluck twist}

In this paper we will use the dotted circle notation of 1-handles. Now let us recall the definition of the Gluck twisting operation:
Let $\varphi: S^2\times S^1\to S^2\times S^1$ be the self-diffeomorphism difined by $\varphi(x, \theta)=(\rho_\theta(x),\theta)$, where $\rho_\theta$ denotes the $\theta$ rotation of the unit 2-sphere about the axis through its poles. 
\begin{definition}[Gluck~\cite{Gl2}. cf.\ \cite{GS}]\label{def:gluck} Let $X$ be a smooth $4$-manifold, and let $S$ be a 2-sphere in $X$ such that its self-intersecion number is zero. Let $X_S$  be the smooth $4$-manifold obtained from $X$ by removing the regular neighborhood $\nu(S)\cong S^2\times D^2$ of $S$, and then  regluing it via $\varphi$. This construction  $X \mapsto X_{S}$ is called the  Gluck twisting of $X$ along $S$.
\end{definition}

There are several ways to draw a handle diagram of Gluck twists, 
here we recall the simple method pointed out in \cite{AY1}. The usefulness of this new method will be evident in the proof of Theorem~\ref{intro:thm}

\vspace{.05in}

Let $X$ and $S$ be as in Definition~\ref{def:gluck}. Consider a handle decomposition
\begin{equation*}
X=\nu(S)\cup \textnormal{handles}, 
\end{equation*}
where we represent $\nu(S)\cong S^2\times D^2$ by a $0$-framed unknot in the handle diagram. 
Introduce a canceling 1- and 2-handle pair locally as in the right picture of Figure~\ref{fig1}. Surger the introduced $S^1\times D^3$ to $S^2\times D^2$, and surger $\nu(S)$ to $S^1\times D^3$ (i.e.\ exchange the dot and the $0$ as in the bottom picture). Note that this is a plug twist along $(W_{-1,0},\, f_{-1,0})$ (see \cite{AY1}).  This operation in fact describes the Gluck twist:
\begin{figure}[ht!]
\begin{center}
\includegraphics[width=3.5in]{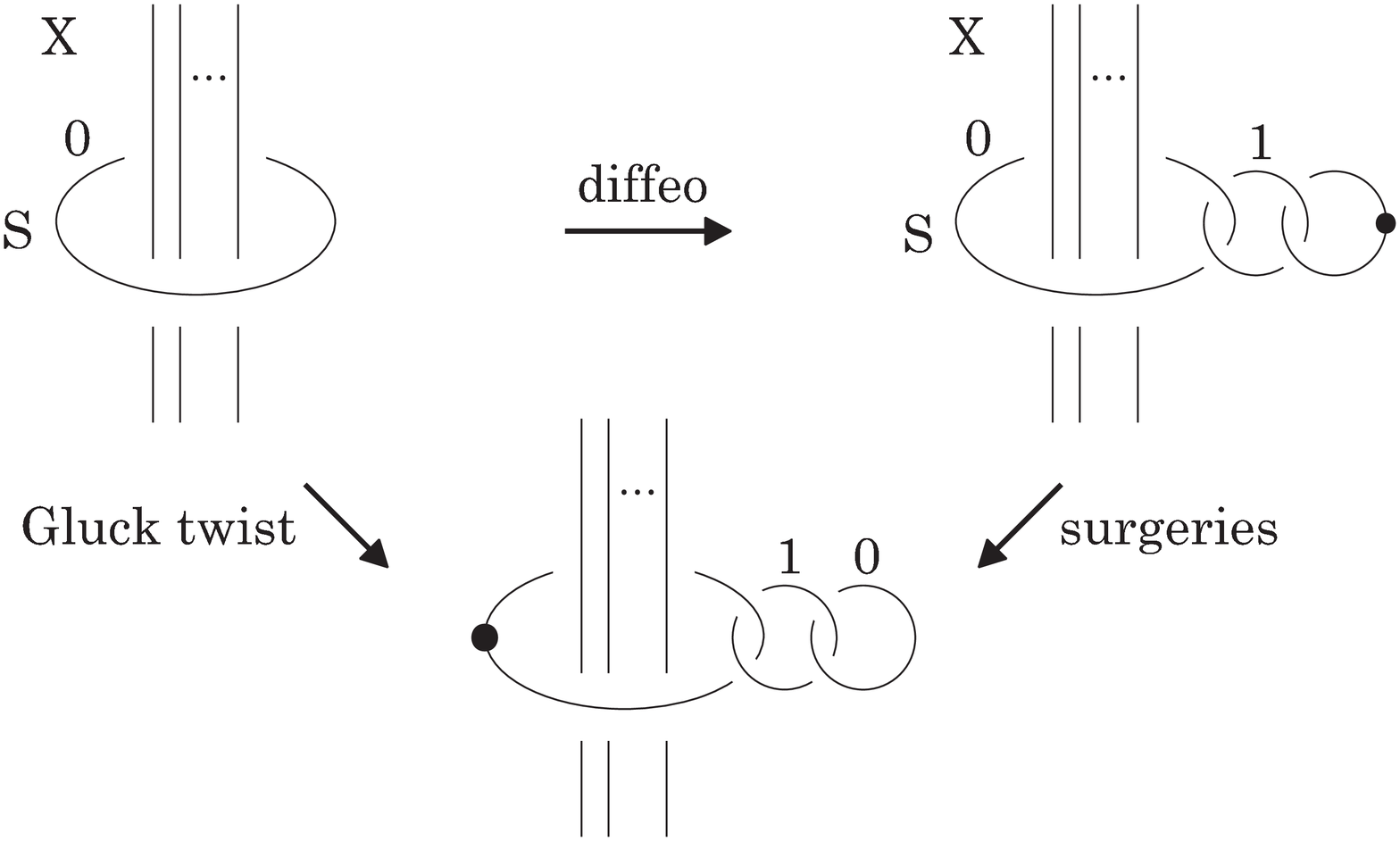}
\caption{Gluck twist to $X$ along $S$}
\label{fig1}
\end{center}
\end{figure}
\begin{lemma}\label{gluck_diagram} The resulting $4$-manifold is diffeomorphic to the Gluck twist $X_S$. If $X$ has boundary, then we may assume that the diffeomorphism is the identity on the boundary. 

\end{lemma}
\begin{proof}One can easily see that the above operation is removing $\nu(S)\cong S^2\times D^2$ and regluing it via a self-diffeomorphism (say $\psi$) of $S^2\times S^1$. A theorem of Gluck \cite{Gl2} (cf.\ \cite{DG}) on the diffeotopy group of $S^2\times S^1$ implies that, we only need to check $\psi$ does not extend to a self-homeomorphism of $S^2\times D^2$. This can be seen by applying the above procedure to the standard handle decomposition of $S^2\times S^2$ and $S^2\times 0$:  the resulting manifold $\mathbb{CP}^2\#\overline{\mathbb{C}\mathbb{P}^2}$ is not homeomorphic to $S^2\times S^2$.  
\end{proof}
\begin{remark}We can similarly prove that the Gluck twist is obtained by the plug twist along $(W_{2m+1,0}, f_{2m+1,0})$ for any integer $m$. \end{remark}
\section{Proof of the theorem}\label{sec:proof}
For a 4-dimentional handlebody $Z$, we denote by $Z^{(i)}$ the subhandlebody of $Z$ consisting of its handles with their indices $\leq  i$. We first prepare the lemma below. 
\begin{lemma}Let $Z$ be a smooth compact connected 4-dimentional handlebody, and let $f:S^2\to Z$ be a continuous map. We denote by $[f]$ the second homology class of $Z$ given by $f$. Suppose $[f]\neq 0$. Then, if necessary after introducing canceling 2- and 3-handle pairs and sliding 2-handles, $[f]$ is represented by a 2-handle of $Z$ whose attaching circle is null-homotoic in $Z^{(1)}$. 
\end{lemma}
\begin{proof}We may assume that $f$ is an immersion by homotopy. Furthermore, we may assume that the image of $f$ is contained in $Z^{(2)}$, since the cocore of a 3-handle is codimention 3. Let $p_1,p_2,\dots, p_n$ be the transverse intersection points of $f(S^2)$ and the cocores of 2-handles of $Z$. Let $D_i$ $(1\leq i\leq n)$ be a small 2-disk neighborhoods of $f^{-1}(p_i)$ in $S^2$. We may assume that each $f(D_i)$ is a core of a 2-handle of $Z$ and that $f(S^2-\coprod_{i=1}^n D_i)$ is contained in $Z^{(1)}$. For each $i$, take an arc $\gamma_i\subset S^2$ connecting $D_i$ and $D_{i+1}$, with $\nu(\gamma_i)$ its tubular neighborhood, so that $\Sigma:= (\cup_{i=1}^n D_i)\cup (\cup_{i=1}^{n-1} \nu(\gamma_i))$ is a 2-disk in $S^2$ and that each $f|_{\nu (\gamma_i)}$ is an embedding. We can isotope $f$ so that each $f(\nu(\gamma_i))$ is contained in $\partial Z^{(1)}$. Now we obtain a 2-handle of $Z$ whose attaching circle is $f(\partial \Sigma)$, by introducing canceling 2- and 3-handle pairs and sliding 2-handles along the band $f(\nu (\gamma_i))$'s. This 2-handle represents the class $[f]$. Since $S^2-\Sigma$ is a 2-disk, the attaching circle $f(\partial \Sigma) = f(\partial (S^2-\Sigma))$ bounds an immersed disk in $Z^{(1)}$. Hence this attaching circle is null-homotopic in $Z^{(1)}$. 
\end{proof} 

We next prove the main theorem. 
Let $K$ be the attaching circle of the handle representing the spherical homology class of $X_{S}^\circ$ with odd self intersection. According to the above lemma, we can assume $K$ is a null homotopic loop in the codimension zero submanifold $Y$ of $X_{S}^\circ$ consisting of its $0$- and $1$-handles. Hence after  (self crossing) clasp moves $K$ can be made to be an unknot. In paricular, if there is another handle represented by a small circle $C$ linking $K$ with zero framing, then by sliding $K$ over $C$ we can make $K$ an unknotted circle as shown in Figure~\ref{figa1}. 
\begin{figure}[ht!]\begin{center}
\includegraphics[width=2.8in]{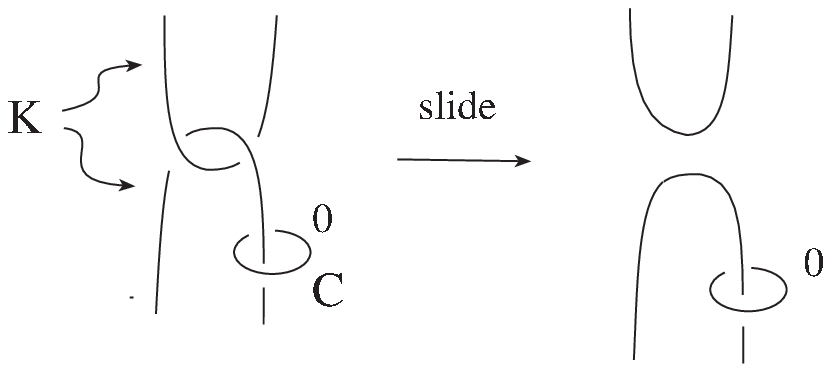}
\caption{Gluck twist to $X$ along $S$}
\label{figa1}
\end{center}
\end{figure}

Having this observation in mind, let us perform the Gluck twisting operation to $X$ along $S$ as shown in Figure~\ref{figa2}. Now by sliding the middle $1$-framed handle over $K$, we get a 2-handle $K_{0}$ with even framing. By using the above observation, we can isotope $K_{0}$ to the standard middle $0$-framed circle of the last picture of Figure~\ref{figa2}. 

\begin{figure}[ht!]
\begin{center}
\includegraphics[width=3.8in]{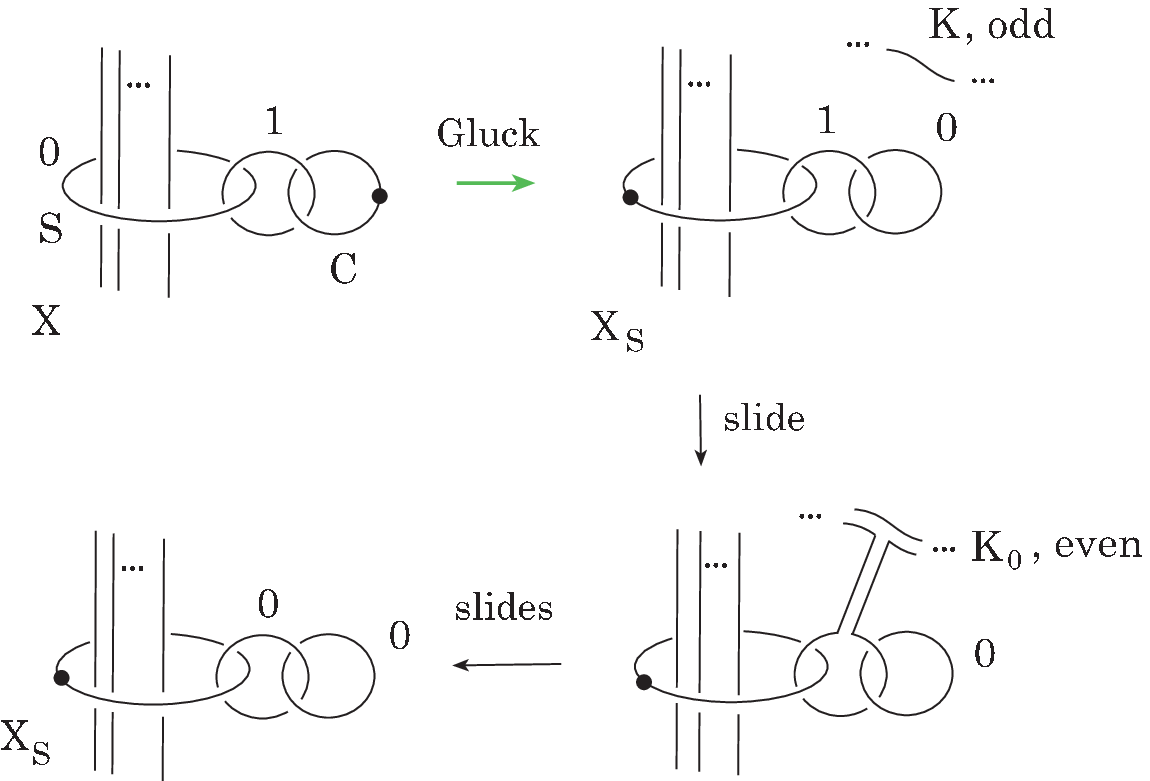}
\caption{Gluck twist to $X$ along $S$}
\label{figa2}
\end{center}
\end{figure}

Here consider the subhandlebody of $X_S$ described in the upper left diagram of Figure~\ref{fig4}.  
We reverse the Gluck twisting operation applied to $X_S$. Lemma~\ref{gluck_diagram} shows that this reversing operation is equivalent to exchanging $0$ and the dot as shown in the left procedure of Figure~\ref{fig4}. The right procedure of Figure~\ref{fig4} shows that this is a composition of undone operations (recall that every self-diffeomorphism of $\partial D^4$ extends to a self-diffeomorphism of $D^4$.). Hence $X_S$ is diffeomorphic to $X$. 
\qed

\begin{figure}[ht!]
\begin{center}
\includegraphics[width=3.4in]{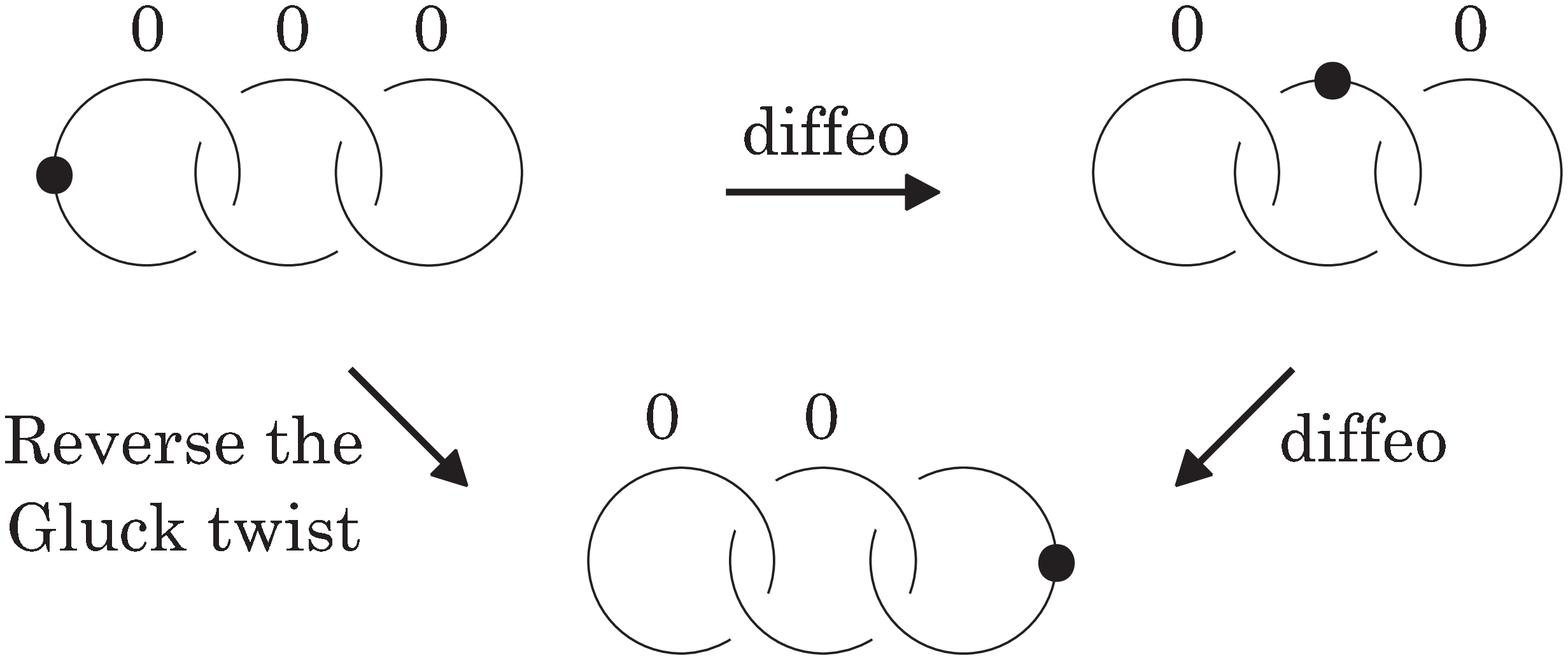}
\caption{Reversing the Gluck twist to $X$ along $S$}
\label{fig4}
\end{center}
\end{figure}



\end{document}